\documentclass[11pt,a4paper,leqno]{amsart}

\usepackage[latin1]{inputenc}
\usepackage[T1]{fontenc}
\usepackage[english]{babel}
\usepackage{amsfonts}
\usepackage{amsmath}
\usepackage{amssymb}
\usepackage{eurosym}
\usepackage{mathrsfs}
\usepackage{palatino}
\usepackage{color}
\usepackage{esint}
\usepackage{url}
\usepackage{verbatim}

\usepackage{enumerate}
\allowdisplaybreaks

\newcommand{\R}{\mathbb{R}}
\newcommand{\C}{\mathbb{C}}

\newcommand{\Z}{\mathbb{Z}}

\newcommand{\bla}{\big \langle}
\newcommand{\bra}{\big \rangle}

\numberwithin{equation}{section}

\newcommand{\ud}[0]{\,\mathrm{d}}

\newcommand{\dist}[0]{\operatorname{dist}}

\newcommand{\BMO}[0]{\operatorname{BMO}}

\newcommand{\supp}[0]{\operatorname{spt}}
\newcommand{\loc}[0]{\operatorname{loc}}

\newcommand{\RH}[0]{\operatorname{RH}}

\renewcommand{\Re}[0]{\operatorname{Re}}

\newcommand{\wt}[1]{{\widetilde{#1}}}

\swapnumbers
\theoremstyle{plain}
\newtheorem{thm}[equation]{Theorem}
\newtheorem{lem}[equation]{Lemma}
\newtheorem{prop}[equation]{Proposition}

\theoremstyle{definition}

\newtheorem{exmp}[equation]{Example}

\theoremstyle{remark}
\newtheorem{rem}[equation]{Remark}

\title[Multilinear commutators in the two-weight setting]{Multilinear commutators in the two-weight setting}

\makeatletter
\@namedef{subjclassname@2010}{%
  \textup{2010} Mathematics Subject Classification}
\makeatother

\subjclass[2010]{42B20}
\keywords{Multilinear Calder\'on-Zygmund operators, commutators, weighted BMO, sparse operators}

\author{Kangwei Li}
\address{Center for Applied Mathematics, Tianjin University, Weijin Road 92, 300072 Tianjin, China}
\email{kli@tju.edu.cn}

\begin{document}

\begin{abstract}
We extend the recently much-studied two-weight commutator estimates to the multilinear setting. In contrast to previous results, our result respects the multilinear nature of the problem fully and is formulated with the genuinely multilinear weights. 
\end{abstract}

\maketitle

\section{Introduction}
The characterization of BMO spaces via the boundedness of commutators has attracted a lot of attention recently. Recall that given a linear operator $T$ and a locally integrable function $b$, the commutator of $T$ and $b$ is defined as $$[b, T]f=bT(f)-T(bf).$$
 This research line initiated from the work of Nehari \cite{Nehari}, who studied the case when the singular integral operator $T$ is the Hilbert transform, by complex analysis methods.  Starting from the celebrated work of Coifman-Rochberg-Weiss \cite{CRW}, commutators have been a central part of harmonic analysis. In this work, they showed that the boundedness of $[b, R_j]$ characterizes the $\BMO$ membership of $b$, 
where $R_j$ is the $j$-th Riesz transform and BMO stands for the usual bounded mean oscillation function space.

Later, in 1976, Muckenhoupt and Wheeden \cite{MW} introduced the weighted BMO space $\BMO_\nu$ with 
  $\nu\in A_\infty$. Although some interesting estimates related with the Hilbert transform are formulated in \cite{MW}, it was until 1985, Bloom \cite{Bloom} finally found the connection between $\BMO_\nu$ and the two-weight boundedness of $[b, H]$:
\[
  \|b\|_{\BMO_\nu} \lesssim \|[b,H]\|_{L^p(\mu) \to L^p(\lambda)} \lesssim \|b\|_{\BMO_\nu}, \qquad 1<p<\infty,
\]where $\mu, \lambda\in A_p$ and $\nu=\mu^{\frac 1p} \lambda^{-\frac 1p}$. Bloom's result extended the commutator estimates to the two-weight context, and from then on Bloom type estimates have been at the focus of harmonic analysts. For the upper bound, Segovia--Torrea \cite{ST} first extended Bloom's result to general Calder\'on-Zygmund operators, see also Holmes--Lacey--Wick \cite{HLW} for a modern proof  and Lerner--Ombrosi--Rivera-R\'ios \cite{LORRadv} for the quantitative upper bounds. The $k$-order iterated case, which is defined  inductively by
\[
C_b^k(T)=[b, C_b^{k-1}(T)], \qquad C_b^1(T)=[b,T]
\] was considered initially with the assumption $b\in \BMO\cap \BMO_\nu$ (see \cite{HW,Hyt16}), and was refined to $\BMO_{\nu^{1/k}}$ later in \cite{LORR}. 
For the lower bound, the influential paper by Holmes--Lacey--Wick \cite{HLW} solved the problem for the Riesz transforms. The general case, i.e. non-degenerate singular integrals, were solved only recently by Hyt\"onen \cite{Hyt18} through the median method. This also improves an earlier result in  \cite{LORR}. 

Recently, Lerner--Ombrosi--Rivera-R\'ios \cite{LORR} and Hyt\"onen \cite{Hyt18} also studied the case when $T$ is the rough homogeneous singular integrals $T_\Omega$.  However, both of them did not provide the upper bound. We will address this problem in Section \ref{sec:rough} as a by-product of our new method. For more about the linear theory, we refer the readers to \cite{A, HPW, LMV:Bloom} and the references therein.

Now it is natural to ask whether one can establish the corresponding multilinear theory.
Given an $n$-linear operator $T$ and a locally integrable function $b$,  we define $$[b, T]_i (f_1, \dots, f_n)=bT(f_1,\dots, f_n)- T(f_1, \dots, bf_i,\dots, f_n),\qquad 1\le i\le n.$$ Notice that by the above definition,
$
[b_2, [b_1, T]_i]_j=[b_1, [b_2, T]_j]_i
$, so given $\mathcal I\subset \{1,\dots, n\}$ we may define the general iterated commutator inductively as
\begin{align}\label{eq:itercommu}
C_{b_{\mathcal I}}^{k_{\mathcal I}}(T):=C_{b_{\{i\}}}^{k_i}(C_{b_{\mathcal I\setminus \{i\}}}^{k_{\mathcal I\setminus \{i\}}}(T)),\quad C_{b_{\{i\}}}^{k_i}(T):=[b_{k_i}^i,\cdots, [b_1^i, T]_i,\cdots, ]_i,\quad  i\in \mathcal I.
\end{align}We simply denote  $C_{b_{\{i\}}}^{k_i}(T)$ by $C_{b}^{k_i}(T)$ when $b=b_1^i=\cdots= b^i_{k_i}$.

In the multilinear setting the most interesting phenomena occurs when the genuinely multilinear weights are used. The genuinely multilinear weights were introduced by Lerner, Ombrosi, P\'erez, Torres and Trujillo-Gonz\'alez in the very influential work \cite{LOPTT}. The point is, one only needs to assume a weaker joint condition on the tuple of weights $(w_1,\cdots, w_n)$ rather than to assume individual conditions on $w_i$. Commutator estimates involving genuinely multilinear weights in the one-weight situation already appear in the literature, see e.g.  \cite{LOPTT}, where they proved
\begin{equation}\label{eq:oneweightcommu}
[b,T]_j: L^{p_1}(w_1)\times \cdots \times L^{p_n}(w_n)\to L^p\big(\prod_{i=1}^n w_i^{\frac p{p_i}}\big)\qquad 1\le j\le n.
\end{equation} 
Multilinear commutator estimates in the two-weight setting have also been studied by Kunwar and Ou in \cite{KO}. However, they did not use the genuinely multilinear weights. To be more precise, when considering e.g. $[b,T]_1$, they need to assume $w_1, \lambda_1\in A_{p_1}$ in addition to natural assumptions. Here, for the first time, we work in the simultaneous presence of both complications.
\begin{thm}\label{thm:main1}
 Let  $T$ be an $n$-linear Calder\'on-Zygmund operator, $1\le i\le n$, $1\le k_i<\infty$ and $C_{b_{\{i\}}}^{k_i}(T)$ be defined as in \eqref{eq:itercommu}. Given $\theta_1,\dots, \theta_{k_i}\in [0,1]$ such that $\sum_{\ell=1}^{k_i}\theta_{\ell}=1$, and let $\theta=\max\{\theta_\ell\}_{1\le \ell\le k_i}$. Let $1<p_1,\dots, p_n<\infty$ and $\frac 1p=\sum_{i=1}^n \frac1{p_i}$. Assume that $(w_1,\dots, w_n), (w_1,\dots, w_{i-1},\lambda_i, w_{i+1},\dots, w_n)\in A_{\vec p}$ with $\nu_i^\theta :=w_i^{\frac \theta{p_i}}\lambda_i^{-\frac \theta{p_i}}\in A_\infty$ and $b^i_{\ell}\in \BMO_{\nu_i^{\theta_{\ell}}}$ for every $1\le \ell\le k_i$. Then
  \begin{align*}
  \Big\| C_{b_{\{i\}}}^{k_i}(T): L^{p_1}(w_1)\times \cdots\times L^{p_n}(w_n)\to L^p\big( \lambda_i^{\frac p{p_i}}\prod_{j\neq i  }w_j^{\frac p{p_j}} \big)\Big\|\lesssim \prod_{\ell=1}^{k_i} \|b_{\ell}^i\|_{\BMO_{\nu_i^{\theta_\ell}}}.
  \end{align*}
\end{thm}

\begin{rem}
Theorem \ref{thm:main1} is the upper bound of the iterated commutator when $\mathcal I$ contains only one element. For general $\mathcal I$, our method also works, but it is more technical and we will record it in Section \ref{sec:proofmain1}, where we also comment on the quantitative upper bounds and a comparison between our bounds and the bounds via the Cauchy integral trick in \cite{DHL,BMMST}.
\end{rem}
\begin{rem}\label{rem:large}
  The assumption $\nu_i^\theta \in A_\infty$ ensures that $\nu_i^{\theta_\ell}\in A_\infty$ for all $1\le \ell \le k_i$ so that the weighted BMO spaces are well-defined. However, if $\theta \le \frac 1n$, then there is no need to assume $\nu_i^\theta \in A_\infty$ as it is automatically true. More details are provided in Subsection \ref{subsec:bmo}.
\end{rem}

Unlike the two-weight situation, the one-weight situation has many tools such as sharp maximal function estimates \cite{PT, PPTT}, Cauchy integral trick \cite{DHL, BMMST} and also sparse domination. 
In the two-weight setting it seems that only sparse domination survives. However, if one simply follows the known strategy, one will meet a term which involves a composition of multilinear sparse operator and linear sparse operator, and this is the reason why the genuinely multilinear weights were not able to appear in the two-weight setting before. See \cite{KO} for the details. 
We overcome this difficulty by establishing a Muckenhoupt-Wheeden type result, then we are able to reduce the problem to estimating the composition of multilinear sparse operator and \textit{weighted} linear sparse operator. The idea behind is based on two facts: 
\begin{enumerate}
\item for any genuinely multilinear weight $(w_1,\cdots, w_n)$, $w_i^{1-p_i'}\in A_\infty$;
\item weighted linear sparse operator is bounded if the underlying weight belongs to $A_\infty$. 
\end{enumerate}  

To complete the multilinear commutator theory in the two-weight setting we still need to provide the lower bounds. It is worth mentioning that, for the lower bound the genuinely multilinear weights were not used even in the one-weight situation in previous results. For example, in \cite{GLW} Guo, Lian and Wu achieved the lower bound when $T$ is certain non-degenerate multilinear Calder\'on-Zygmund operator. Kunwar and Ou \cite{KO} also obtained a similar result for the multilinear Haar multipliers. However, both of them need to assume individual conditions on $w_i$.

In contrast to previous results, we are able to use genuinely multilinear weights in the two-weight setting.
\begin{thm}\label{thm:main2}
 Let  $T$ be an $n$-linear non-degenerate Calder\'on-Zygmund operator (see Subsection \ref{subsec:czo} for the definition), $1\le i\le n$, $1\le k_i<\infty$, $b\in L_{\loc}^{k_i}$ and $C_{b}^{k_i}(T)$ be defined as in \eqref{eq:itercommu}. Let $1<p_1,\dots, p_n<\infty$ and $\frac 1p=\sum_{i=1}^n \frac1{p_i}$. Assume that $$(w_1,\dots, w_n), (w_1,\dots, w_{i-1},\lambda_i, w_{i+1},\dots, w_n)\in A_{\vec p}$$ with $\nu_i^{\frac 1{k_i}} :=w_i^{\frac 1{k_ip_i}}\lambda_i^{-\frac 1{k_ip_i}}\in A_\infty$. Then
  \begin{align*}
   \|b \|_{\BMO_{\nu_i^{1/{k_i}}}}^{k_i}\lesssim \Big\| C_{b}^{k_i}(T): L^{p_1}(w_1)\times \cdots\times L^{p_n}(w_n)\to L^p\big( \lambda_i^{\frac p{p_i}}\prod_{j\neq i  }w_j^{\frac p{p_j}} \big)\Big\|.
  \end{align*}
\end{thm}
\begin{rem}\label{rem:weakerassumption}
Similar as Remark \ref{rem:large}, if $k_i\ge n$ then we do not need to assume $\nu_i^{\frac 1{k_i}}\in A_\infty$ as it is automatically true. Moreover, in Section 4 we will actually prove the above result with a weaker boundedness assumption.
\end{rem}
This paper is organized as the following: in Section \ref{sec:lemmata} we provide necessary notations and auxiliary results. Section \ref{sec:proofmain1} is devoted to proving Theorem \ref{thm:main1}. The lower bound is handled is Section \ref{sec:proofmain2}. In Section \ref{sec:rough} we provide a short discussion about the upper bound of the commutator of rough homogeneous singular integrals, and general sparse operators involving  weighted BMO functions. 

\subsection*{Acknowledgements}
This work was supported by the National Natural Science Foundation of China through project number 12001400. The author also would like to thank the anonymous referee for his/her careful reading that helped improving the presentation of the paper.

\section{Definitions and auxiliary lemmata}\label{sec:lemmata}
\subsection{Basic notations}
We denote $A\lesssim B$ if $A\le CB$ for some constant that can depend on the dimension, Lebesgue exponents, weight constants, and on various other constants appearing in the assumptions. We denote $A\sim B$ if $B\lesssim A\lesssim B$. 

Given a cube $Q$, a measure $\mu$ and a locally integrable function $f$, we denote the average $\mu(Q)^{-1}\int_Q f\ud \mu=f_Q^\mu=\langle f\rangle_Q^\mu$. When $\mu$ is Lebesgue measure we simply write $|Q|^{-1}\int_Q f=\fint_Q f =f_Q =\langle f\rangle_Q$. 
\subsection{Multilinear Calder\'on-Zygmund operators}\label{subsec:czo}The multilinear Calder\'on-Zygmund theory was systematically formulated by Grafakos and Torres \cite{GT}. Let us begin with the definition of multilinear Calder\'on-Zygmund operators (CZOs).
Let $\Delta:=\{(x, y_1,\cdots, y_n)\in (\R^{d})^{n+1}: x=y_1=\cdots =y_n\}$ be the diagonal in $(\R^{d})^{n+1}$. We say $K: (\R^{d})^{n+1}\setminus \Delta \to \C$ is a multilinear Calder\'on-Zygmund kernel if
\begin{align}
  |K(x,y)|&\le \frac{C}{(\sum_{i=1}^n|x-y_i|)^{nd}}, \label{eq:size}\\
   |K(x+h,y)-K(x,y)|+ \sum_{i=1}^{n}&|K(x,\cdots, y_i+h,\cdots)-K(x,y)| \label{eq:regularity}\\
   &\le \frac{C}{(\sum_{i=1}^n|x-y_i|)^{nd}}\omega\left(\frac{h}{\sum_{i=1}^n|x-y_i|}\right)\nonumber
\end{align}
whenever $h\le \frac 12 \max_i |x-y_i|$, where $\omega$ is an increasing subadditive function with $\omega(0)=0$ and
\[
  \|\omega\|_{\rm{Dini}}=\int_{0}^{1}\omega(t)\frac{\ud t}{t}<\infty.
\]

Then we say $T$ is a multilinear CZO if $T$ is initially bounded from $L^{q_1}(\R^d)\times \cdots \times L^{q_n}(\R^d)\to L^q(\R^d)$ with $q_i\in (1,\infty]$, $\frac 1q=\sum_{i=1}^n \frac 1{q_i}>0$ and there exists a multilinear Calder\'on-Zygmund kernel $K$ such that for all $f_1,\dots, f_n\in C_c^\infty(\R^d)$,
\begin{equation}\label{eq:kernelrep}
  T(f_1,\dots, f_n)(x)=\int_{\R^{nd}} K(x,y_1,\cdots, y_n) \prod_{i=1}^n f_i(y_i) \ud y,\qquad x\notin \bigcap_{i=1}^n \supp f_i.
\end{equation}
In \cite{DHL}, the author, Dami\'an and Hormozi obtained a pointwise sparse bound for multilinear CZOs introduced in the above. That is,
\[
  |T(f_1,\dots, f_n)(x)|\lesssim \sum_{j=1}^{3^d}\mathcal A_{S_j}(f_1,\dots, f_n)(x):=\sum_{j=1}^{3^d}\sum_{Q\in \mathcal S_j} \prod_{i=1}^{n} \langle |f_i|\rangle_Q 1_Q(x),
\]where  $\mathcal S_j$ is a sparse collection for every $1\le j\le 3^d$. Recall that we say a collection of cubes $\mathcal S$ is $\rho$-sparse if for every $S\in \mathcal S$, there exists $E_S\subset S$ with $|E_S|\ge \rho |S|$ and $\{E_S\}$ are pairwise disjoint. Usually we choose $\rho=\frac 12$.

Now we define the non-degenerate Calder\'on-Zygmund operators. We say $T$ is a non-degenerate multilinear CZO if there is a function $K$ such that \eqref{eq:size}, \eqref{eq:regularity} and \eqref{eq:kernelrep} holds with $\omega(0)\to 0$ when $t\to 0$, and in addition, for every $y\in \R^d$ and $r>0$, there exists $x\notin B(y, r)$ with
\begin{equation}\label{eq:non-degen}
|K(x,y,\dots,y)|\gtrsim \frac 1{r^{nd}}.
\end{equation}
Note that \eqref{eq:size} and \eqref{eq:non-degen} imply that $|x-y|\sim r$. The multilinear Riesz transforms are typical examples of non-degenerate Calder\'on-Zygmund operators. 
\subsection{Weights}
By weights we mean positive locally integrable functions. Recall that we say a weight $w\in A_p$ if
\[
 [w]_{A_p}:= \sup_Q \fint_Q w \left(\fint_Q w^{-\frac 1{p-1}}\right)^{p-1}<\infty, \qquad 1<p<\infty.
\]And we say $w\in A_\infty$ if
\[
  [w]_{A_\infty}:= \sup_Q \frac 1{w(Q)}\int_Q M(w\chi_Q)<\infty.
\]An important property of $A_\infty$ weights is recorded as the following lemma.
\begin{lem}\label{lem:RHI}
  Let $t_1,\dots, t_n\in (0,\infty)$ and $w_1,\dots, w_n\in A_{\infty}$. Then there exists a constant $\RH_{\vec w,\vec t}$ such that for every cube $Q$,
  \begin{equation}\label{eq:RHI}
     \prod_{i=1}^n \left(\fint_Q w_i \right)^{t_i}\le \RH_{\vec w,\vec t} \fint_Q \prod_{i=1}^n w_i^{t_i}.
  \end{equation}We may abuse notation a little that we still denote by $ \RH_{\vec w,\vec t}$  the best constant such that \eqref{eq:RHI} holds.
\end{lem}
\begin{rem}
  The case when $\sum_i t_i\le 1$ was obtained first by Xue and Yan \cite{XY}, and rediscovered recently by Cruz-Uribe and Moen \cite{CUM}. They also referred it as the multilinear Reverse H\"older property. The case $\sum_i t_i> 1$ is actually a simple consequence of the aforementioned result and H\"older's inequality. However, for sake of completeness we provide a direct proof in below. 
\end{rem}
\begin{proof}[Proof of Lemma \ref{lem:RHI}]
  By embedding of $A_\infty$ weights, there exist $q_1,\dots, q_n\in (1,\infty)$ such that $w_i\in A_{q_i}$, $1\le i \le n$. Fix a cube $Q$, let
  \[
    E_i:= \left\{x\in Q: w_i(x)< \frac{\langle w_i\rangle_Q}{(2n)^{q_i-1}[w_i]_{A_{q_i}}}\right\}.
  \]
  Then by the definition of $A_{q_i}$ weights, we have
  \begin{align}\label{eq:smallsize}
    \frac{|E_i|}{|Q|}  \le \frac 1{2n [w_i]_{A_{q_i}}^{\frac 1{q_i-1}}} \langle w_i\rangle_Q^{\frac 1{q_i-1}}\left(\fint_Q w_i^{-\frac 1{q_i-1}} \right)\le \frac 1{2n}.
  \end{align}
  Let $E:=\cup_{i=1}^{n} E_i$ and $F=Q\setminus E$, we have by \eqref{eq:smallsize} that $|F|\ge \frac 12 |Q|$.
  Thus
  \begin{align*}
   \prod_{i=1}^n \left(\fint_Q w_i \right)^{t_i}&\le  \left(\prod_{i=1}^n (2n)^{(q_i-1)t_i}[w_i]_{A_{q_i}}^{t_i}\right) \fint_F \prod_{i=1}^n  w_i^{t_i}\\&\le 2 \left(\prod_{i=1}^n (2n)^{(q_i-1)t_i}[w_i]_{A_{q_i}}^{t_i}\right) \fint_Q \prod_{i=1}^n  w_i^{t_i}.
  \end{align*}
\end{proof}
When $\sum_{i=1}^n t_i>1$, the converse of \eqref{eq:RHI} cannot be obtained through H\"older's inequality. Nevertheless, we record the following lemma, which is useful in the proof of our main results.
\begin{lem}\label{lem:converse}
  Let $t_1,\dots, t_n\in (0,\infty)$ and $w_1,\dots, w_n$ be weights such that $\prod_{i=1}^n w_i^{t_i}\in A_\infty$. Then there is a constant $K$ depending on $[\prod_{i=1}^n w_i^{t_i}]_{A_\infty}$  such that for every cube $Q$,
  \[
    \fint_Q \prod_{i=1}^n w_i^{t_i}\le K \prod_{i=1}^n \left(\fint_Q w_i \right)^{t_i}.
  \]
\end{lem}
\begin{proof}
  Fix a cube $Q$. Let $E_i:=\{x\in Q: w_i(x)> 2n \langle w_i\rangle_Q\}$. Then by Chebyshev it is easy to see that $|E_i|\le \frac 1{2n}|Q|$. Let $F=Q\setminus \cup_{i=1}^n E_i$, then $|F|\ge \frac 12|Q|$ and therefore,
  \begin{align*}
    \int_Q \prod_{i=1}^n w_i^{t_i} & \lesssim_{[\prod_{i=1}^n w_i^{t_i}]_{A_\infty}}  \int_F \prod_{i=1}^n w_i^{t_i}\\&\le (2n)^{\sum_{i=1}^n t_i} |F|\prod_{i=1}^n \left(\fint_Q w_i \right)^{t_i}\le(2n)^{\sum_{i=1}^n t_i} |Q|\prod_{i=1}^n \left(\fint_Q w_i  \right)^{t_i}.
  \end{align*}We are done.
\end{proof}

In the multilinear setting, Lerner, Ombrosi, P\'erez, Torres and Trujillo-Gonz\'alez \cite{LOPTT} introduced the multiple $A_{\vec p}$ weights: we say $\vec w:=(w_1,\dots, w_n)\in A_{\vec p}$ if
\[
  [\vec w]_{A_{\vec p}}:=\sup_Q \left(\fint_Q \prod_{i=1}^{n} w_i^{\frac p{p_i}}\right)\prod_{i=1}^{n}\left(\fint_Q w_i^{1-p_i'}\right)^{\frac{p}{p_i'}}<\infty.
\]It is proved in \cite{LOPTT} that $\vec w\in A_{\vec p}$ if and only if
\begin{equation}\label{eq:multipleweights}
 w:=\prod_{i=1}^{n} w_i^{\frac p{p_i}}\in A_{np}\qquad \mbox{and}\qquad w_i^{1-p_i'}\in A_{np_i'}, \qquad 1\le i \le n.
\end{equation}
We shall use \eqref{eq:multipleweights} frequently.

For our purpose we also record the following result obtained in \cite{LMS} by the author, Moen and Sun.
\begin{prop}\label{prop:lms}
Let $\mathcal A_{\mathcal S}$ be a multilinear sparse operator, $1<p_1,\dots, p_n<\infty$ with $1/p=\sum_{i=1}^n 1/{p_i}$ and $\vec w\in A_{\vec p}$. Then
\[
  \|\mathcal A_{\mathcal S}\|_{L^{p_1}(w_1)\times \cdots \times L^{p_n}(w_n)\to L^p(w)}\le C_{n,d,\vec p,T} [\vec w]_{A_{\vec p}}^{\max \{1, \frac{p_1'}{p},\cdots, \frac{p_n'}{p}\}}.
\]In particular, when $p\le 1$, the following stronger estimate holds
\[
 \Big\| \Big(\sum_{Q\in \mathcal S}\prod_{i=1}^n \langle |f_i|\rangle_Q^p 1_Q\Big)^{\frac 1p}\Big\|_{L^p(w)}\le  C_{n,d,\vec p,T} [\vec w]_{A_{\vec p}}^{\max \{\frac{p_1'}{p},\cdots, \frac{p_n'}{p}\}}\prod_{i=1}^{n}\|f_i\|_{L^{p_i}(w_i)}.
\]
\end{prop}

\subsection{BMO spaces}\label{subsec:bmo}
 Let $\nu\in A_\infty$, we say $b\in\BMO_\nu$ if
\[
\|b\|_{\BMO_{\nu}}:=\sup_{Q}\frac {1}{\nu(Q)}\int_Q |b-b_Q|<\infty.
\]In practical cases, $\nu=\mu^{\frac 1p}\lambda^{-\frac 1p}$. In the linear case, there is no need to a priori assume that  $\nu\in A_\infty$ because we always have $\nu \in A_2$, which is easy to verify since $\mu, \lambda \in A_p$. However, in the multilinear case, by \eqref{eq:multipleweights} we only have $\mu^{1-p'}, \lambda^{1-p'}\in A_{np'}$, which implies that $\nu^{\frac 1n}\in A_\infty$, but in general one cannot deduce that $\nu\in A_\infty$-- it can be even not locally integrable. This is easily seen by the following example.
\begin{exmp}\label{exmp:nuainfty}
  For simplicity we provide an example for $n=2$ and $d=1$. Let $(p_1, p_2)=(2,4)$ and $(w_1, w_2)=(|x|^{-2}, |x|^2)$, $(\lambda_1, w_2)=(1, |x|^2)$. It is easy to check that
  \[
    w_1^{\frac p{p_1}} w_2^{\frac p{p_2}}=|x|^{-\frac 23}\in A_{2p}, \,\, w_1^{1-p_1'}=|x|^{2}\in A_{2p_1'},\,\, w_2^{1-p_2'}=|x|^{-\frac 23}\in A_{2p_2'}
  \]and $\lambda_1^{\frac p{p_1}} w_2^{\frac p{p_2}}=|x|^{\frac 23}\in A_{2p}$. With these facts and \eqref{eq:multipleweights} we know that $(w_1, w_2), (\lambda_1, w_2)\in A_{\vec p}$. However, $\nu=w_1^{\frac 1{p_1}}\lambda_1^{-\frac 1{p_1}}=|x|^{-1}\notin L_{\loc}^1(\R)$.
\end{exmp}
Example \ref{exmp:nuainfty} shows that it is reasonable to a priori assume that $\nu \in A_\infty$. In fact, only if $\nu \in A_\infty$ something interesting can happen, for example, Wu \cite{Wu} showed that when $\nu\in A_\infty$ the dual space of $\BMO_\nu$ is
the weighted Hardy space $H^1(\nu)$. Moreover, it was proved by Muckenhoupt and Wheeden that for several interesting estimates related with $b\in\BMO_\nu$, it is necessary to assume $\nu\in A_\infty$ (see \cite{MW}).

 We also define the weighted Bloom $\BMO$: we say $b\in \BMO_\nu(\sigma)$ if
\[
\|b\|_{\BMO_{\nu}(\sigma)}:=\sup_{Q}\frac {1}{\nu\sigma(Q)}\int_Q |b-b_Q| \sigma<\infty.
\]
Recall that when $\nu=1$, it is a classical result by Muckenhoupt and Wheeden \cite{MW} that for any $\sigma\in A_\infty$, one has
\[
\BMO_1=\BMO_1(\sigma).
\]
It would be natural to expect that in the Bloom $\BMO$ setting, one has similar result. We shall show that this is indeed the case, with some natual assumption on the weights.
\begin{lem}\label{lem:equivalence}
Let $\nu, \sigma\in A_\infty$. If $\nu\sigma\in A_\infty$, then $\BMO_\nu=\BMO_\nu(\sigma)$.
\end{lem}We will see that in practical cases, we always have $\nu \sigma\in A_\infty$ for free.
To prove this lemma, we shall need the following proposition, whose proof is standard, but for sake of completeness, we give the full details.
\begin{prop}\label{prop:sparsecontrol}
Let $\sigma\in A_\infty$. Then there holds
\begin{align}\label{eq:sparse}
 |b-b_{Q_0}^\sigma|  1_{Q_0}\lesssim  \sum_{Q\in \mathcal S(Q_0)}\langle | b-b_Q^\sigma|\rangle_Q^\sigma 1_Q,
\end{align}where $\mathcal S(Q_0)$ is a sparse collection with respect to $\sigma$ and with all elements  contained in $Q_0$.
\end{prop}
\begin{proof}
Let
$\alpha=2\langle |b- b_{Q_0}^\sigma| \rangle_{Q_0}^\sigma
$. Form the Calder\'on-Zygmund decomposition of $|b-b_{Q_0}^\sigma|  1_{Q_0}$ at height $\alpha$ with respect to $\sigma$, we get a collection of maximal cubes $\{Q_j\}$ in $\mathcal D(Q_0)$ with the property that
\[
\langle | b- b_{Q_0}^\sigma| \rangle_{Q_j}^\sigma >\alpha.
\]Denote $E:=Q_0\setminus \cup_j Q_j$, we have $|b-b_{Q_0}^\sigma|  1_{E}\le \alpha$.
By maximality,
\begin{equation}\label{eq:maxi}
\langle | b-b_{Q_0}^\sigma|\rangle_{Q_j^{(1)}}^\sigma\le \alpha.
\end{equation}Hence
\begin{align*}
\sum_{j}\sigma(Q_j)&<\alpha^{-1} \sum_{j} \int_{Q_j} |b-b_{Q_0}^\sigma|\sigma
\le \frac 12 \sigma(Q_0).
\end{align*}Finally we are able to write
\begin{align*}
 |b-b_{Q_0}^\sigma|  1_{Q_0}&\le  |b-b_{Q_0}^\sigma|  1_{E}+ \sum_j |b-b_{Q_0}^\sigma|  1_{Q_j}\\
 &\le \alpha 1_E+\sum_{j}|b_{Q_j}^\sigma -b_{Q_0}^\sigma|1_{Q_j}+ \sum_{j}|b-b_{Q_j}^\sigma |1_{Q_j}\\
 &\le (D_\sigma +1) \alpha 1_{Q_0} + \sum_{j}|b-b_{Q_j}^\sigma |1_{Q_j},
\end{align*}where in the last step we have used \eqref{eq:maxi} and the doubling property of $\sigma$ (the doubling constant depends on $[\sigma]_{A_\infty}$).
Then \eqref{eq:sparse} is concluded by the above recursive inequality.
\end{proof}
\begin{rem}\label{rem:altdef}
  Proposition \ref{prop:sparsecontrol} provides an alternative definition of $\BMO_\nu(\sigma)$, that is, we may define 
  \[
   \|b\|_{ \wt{\BMO}_\nu(\sigma)}:=\sup_{Q}\frac {1}{\nu\sigma(Q)}\int_Q |b-b_Q^\sigma| \sigma.
  \]
  Indeed, $ \|b\|_{ \wt{\BMO}_\nu(\sigma)}\lesssim  \|b\|_{ {\BMO}_\nu(\sigma)}$ is trivial. So we only show the other direction. We have
  \begin{align*}
 \frac {1}{\nu\sigma(Q)}\int_Q |b-b_Q| \sigma\le \frac {1}{\nu\sigma(Q)}\int_Q |b-b_Q^\sigma| \sigma+ \frac {\sigma(Q)}{\nu\sigma(Q)} |b_Q-b_Q^\sigma|.
  \end{align*}  
  By Proposition \ref{prop:sparsecontrol} we have 
  \begin{align*}
|b_Q-b_Q^\sigma|\le \frac 1{|Q|}\int_Q |b-b_Q^\sigma|&\lesssim \|b\|_{ \wt{\BMO}_\nu(\sigma)}\frac 1{|Q|}\sum_{P\in \mathcal S(Q)}\frac{\nu\sigma(P)}{\sigma(P)}|P|\\
&\lesssim  \|b\|_{ \wt{\BMO}_\nu(\sigma)}\frac 1{|Q|}\sum_{P\in \mathcal S(Q)}\nu(P)\lesssim \|b\|_{ \wt{\BMO}_\nu(\sigma)} \langle \nu\rangle_Q,
  \end{align*}where we have used Lemma \ref{lem:converse} in the third inequality. The argument is concluded now by using Lemma \ref{lem:RHI}. 
\end{rem}

\begin{proof}[Proof of Lemma \ref{lem:equivalence}]
We first show that $\BMO_\nu \subset \BMO_\nu(\sigma)$. Let $b\in \BMO_\nu$. Fix a cube $Q_0$, by Proposition \ref{prop:sparsecontrol},
\[
|b-b_{Q_0}|1_{Q_0} \lesssim \sum_{Q\in \mathcal S(Q_0)}\langle | b-b_Q|\rangle_Q 1_Q,
\] where $\mathcal S(Q_0)$ is a sparse family contained in $Q_0$. Then we have
\begin{align*}
\frac {1}{\nu\sigma(Q_0)}\int_{Q_0} |b-b_{Q_0}| \sigma &\lesssim \|b\|_{\BMO_\nu}\frac {1}{\nu\sigma(Q_0)} \sum_{Q\in \mathcal S(Q_0)}\langle \nu\rangle_Q \sigma(Q)\\
&\lesssim   \|b\|_{\BMO_\nu}\frac {1}{\nu\sigma(Q_0)} \sum_{Q\in \mathcal S(Q_0)}\nu\sigma(Q)\lesssim  \|b\|_{\BMO_\nu},
\end{align*}where in the second inequality we have used Lemma \ref{lem:RHI}, and in the last step we have used $\nu\sigma\in A_\infty$. It remains to prove the converse direction. Let $b\in \BMO_\nu(\sigma)$ and fix a cube $Q_0$. By \eqref{eq:sparse}, we have
\begin{align*}
\frac 1{\nu(Q_0)}\int_{Q_0} |b- b_{Q_0}| & \le \frac 1{\nu(Q_0)}\int_{Q_0} |b- b_{Q_0}^\sigma|+\frac{|Q_0|}{\nu(Q_0)}|b_{Q_0}-b_{Q_0}^\sigma|\\
&\le  \frac 2{\nu(Q_0)}\int_{Q_0} |b- b_{Q_0}^\sigma|\\
&\lesssim \frac 2{\nu(Q_0)}\sum_{Q\in \mathcal S(Q_0)} \langle | b-b_Q^\sigma|\rangle_Q^\sigma |Q|\\
&\le \|b\|_{\BMO_\nu(\sigma)}\frac 4{\nu(Q_0)}\sum_{Q\in \mathcal S(Q_0)} \frac{\nu\sigma(Q)}{\sigma(Q)}|Q|,
\end{align*}
where we have used the fact that
\[
 \langle | b-b_Q^\sigma|\rangle_Q^\sigma\le  \langle | b-b_Q|\rangle_Q^\sigma + \langle | b_Q-b_Q^\sigma|\rangle_Q^\sigma\le 2 \langle | b-b_Q|\rangle_Q^\sigma.
\]
Now since $\nu\sigma\in A_\infty$, by Lemma \ref{lem:converse} we have
\[
\langle \nu\sigma \rangle_Q \lesssim  \langle \nu\rangle_Q \langle \sigma \rangle_Q .
\]Hence
\begin{align*}
\sum_{Q\in \mathcal S(Q_0)} \frac{\nu\sigma(Q)}{\sigma(Q)}|Q|\lesssim   \sum_{Q\in \mathcal S(Q_0)} \nu(Q)\lesssim  \nu(Q_0)
\end{align*}and we are done.
\end{proof}
John-Nirenberg inequality is a key feature of BMO spaces, which has independent interest. We record the John-Nirenberg inequality in our setting as the following:
\begin{thm}\label{thm:jn}
Given $\mu, \lambda$ and $1<q<\infty$ such that $\eta:= \lambda^{1-q'}, \sigma:=\mu^{1-q'}\in A_\infty$. Let $\nu=\mu^{\frac 1q}\lambda^{-\frac 1q}$. Then
\[
\sup_Q\left( \frac 1{\eta(Q)}\int_Q |b-b_Q|^{q'}\sigma\right)^{\frac 1{q'}}\sim \|b\|_{\BMO_{\nu}}.
\]
\end{thm}
\begin{proof}
We first prove the $\gtrsim$ direction. Indeed, by H{\"o}lder's inequality, we see that
\[
\frac 1{\nu\sigma(Q)} \int_Q |b-b_Q|\sigma \le \frac 1{\nu\sigma(Q)} \Big(\int_Q |b-b_Q|^{q'}\sigma \Big)^{\frac 1{q'}}\sigma(Q)^{\frac 1q}\lesssim \left( \frac 1{\eta(Q)}\int_Q |b-b_Q|^{q'}\sigma\right)^{\frac 1{q'}},
\]
where in the last step we have used by Lemma \ref{lem:RHI} that
\[
\langle \sigma \rangle_Q^{\frac 1q}\langle \eta\rangle_Q^{\frac 1{q'}}\lesssim  \big\langle \sigma^{\frac 1q}\eta^{\frac 1{q'}}\big\rangle_Q=  \langle \nu\sigma \rangle_Q.
\]
Next we prove the $\lesssim$ direction. Again, we use
\begin{align*}
|b-b_Q|1_Q \lesssim \sum_{I\in \mathcal S(Q)}\langle |b-b_I|\rangle_I 1_I\le \|b\|_{\BMO_\nu}\sum_{I\in \mathcal S(Q)}\langle \nu\rangle_I 1_I\lesssim  \|b\|_{\BMO_\nu}\sum_{I\in \mathcal S(Q)}\langle \nu\rangle_I^\sigma 1_I,
\end{align*}where in the last step we have used
\[
\langle \nu \rangle_I \lesssim \frac{\langle \nu \sigma \rangle_I}{\langle \sigma\rangle_I},
\]which is again, due to Lemma \ref{lem:RHI}.
Since $\sigma\in A_\infty$, the sparseness with respect to Lebesgue measure is equivalent with the sparseness with respect to $\sigma$. Thus using the boundedness of sparse operator we have
\begin{align*}
\int_Q |b-b_Q|^{q'} \sigma \lesssim \|b\|_{\BMO_\nu}^{q'}\int_Q \nu^{q'} \sigma= \|b\|_{\BMO_\nu}^{q'}\eta(Q)
\end{align*}and we are done.
\end{proof}
\section{Proof of Theorem \ref{thm:main1}}\label{sec:proofmain1}
Our proof is based on sparse domination technique. For this purpose we record the following result, which is more general than \cite[Proposition 2.1]{KO}.
\begin{prop}\label{prop:sparsedom}
 Let  $T$ be an $n$-linear Calder\'on-Zygmund operator, $\mathcal I \subset \{1,\dots, n\}$, $1\le k_i<\infty$ and $C_{b_{\mathcal I}}^{k_{\mathcal I}}(T)$ be defined as in \eqref{eq:itercommu}. Let $f_1,\dots, f_n$ be compactly supported functions. Then there exist $3^d$ sparse collections $\mathcal S_m$, $1\le m\le 3^d$, such that
 \begin{align*}
 |C_{b_{\mathcal I}}^{k_{\mathcal I}}(T)(f_1,\dots, f_n)|\lesssim \sum_{m=1}^{3^d}\sum_{\vec \gamma \in \{1,2\}^{L}}\mathcal A_{\mathcal S_m, b_{\mathcal I}}^{\vec \gamma}(f_1,\dots, f_n),
 \end{align*}where $L=\sum_{i\in \mathcal I}k_i$ and for fixed $\vec \gamma$,
\begin{align*}
& \mathcal A_{\mathcal S_m, b_{\mathcal I}}^{\vec \gamma}(f_1,\dots, f_n)\\&=\sum_{Q\in \mathcal S_m} \prod_{i\in \mathcal I}   \Big\langle  |f_i|\prod_{\ell_i\in B_i}  |b^i_{\ell_i}-\langle b^i_{\ell_i}\rangle_Q|\Big\rangle_Q \Big(\prod_{j\notin \mathcal I}\langle |f_j|\rangle_Q \Big)\prod_{i\in \mathcal I}\prod_{\ell_i\in A_i}  |b^i_{\ell_i}-\langle b^i_{\ell_i}\rangle_Q| 1_Q,
\end{align*}where
\[
A_i=\{\ell_i: \gamma_{\ell_i}=1\},\qquad B_i=\{\ell_i: \gamma_{\ell_i}=2\}.
\]
\end{prop}
The proof of Proposition \ref{prop:sparsedom} is similar as \cite{KO}, by employing the idea in \cite{IR}. Thus we omit the details. Before we begin the proof, we also need the following result stated in \cite[Lemma 5.1]{LORRadv}, which is a stronger version than Proposition \ref{prop:sparsecontrol} in the unweighted case.
\begin{lem}\label{lem:uniformsparse}
Let $\mathcal S$ be a $\gamma$-sparse collection of dyadic cubes and $b\in L_{\loc}^1$. Then there exists a $\frac \gamma {2(1+\gamma)}$ -sparse family $\wt {\mathcal S}$ such that $\wt {\mathcal S}\supset \mathcal S$ and for all $Q\in \wt {\mathcal S}$,
\[
|b-b_{Q}|  1_{Q}\lesssim \sum_{P\in \wt{\mathcal S}, P\subset Q}\langle | b-b_P|\rangle_P 1_P,
\]
\end{lem}
Now we start the proof of Theorem \ref{thm:main1}. Fix $\vec \gamma$. By Proposition \ref{prop:sparsedom} it suffices to bound
\[
\Big\|\sum_{Q\in \mathcal S} \Big\langle  |f_i|\prod_{\ell\in B_i}  |b^i_{\ell}-\langle b^i_{\ell}\rangle_Q|\Big\rangle_Q \Big(\prod_{j\neq i}\langle |f_j|\rangle_Q \Big) \prod_{\ell\in A_i}  |b^i_{\ell}-\langle b^i_{\ell}\rangle_Q| 1_Q\Big\|_{L^p(\lambda_i^{\frac p{p_i}}\prod_{j\neq i}w_j^{\frac p{p_j}})}.
\]
Let us first consider the case $p\le 1$. By Lemma \ref{lem:uniformsparse}, we have
\[
 |b^i_{\ell}-\langle b^i_{\ell}\rangle_Q| 1_Q\lesssim \sum_{P_{\ell}\in \mathcal S_{\ell}, P_\ell \subset Q}\langle  |b^i_{\ell}-\langle b^i_{\ell}\rangle_{P_{\ell}}|\rangle_{P_{\ell}} 1_{P_{\ell}}\le  \|b_{\ell}^i\|_{\BMO_{\nu_i^{\theta_\ell}}}\sum_{P_{\ell}\in \mathcal S_{\ell}, P_\ell \subset Q} \langle \nu_i^{\theta_{\ell}}\rangle_{P_{\ell}} 1_{P_{\ell}},
\]where we may assume (by splitting $\mathcal S$ if necessary) that $\mathcal S_{\ell}$ is $\rho$-sparse with $\rho<k_i^{-1}$.
Now that $\wt{\mathcal S} = \cup_{1\le \ell\le k_i}\mathcal S_{\ell} $ is still sparse, we have
\begin{align*}
\int_Q  \Big(\prod_{\ell\in A_i} & |b^i_{\ell}-\langle b^i_{\ell}\rangle_Q|^p\Big) \lambda_i^{\frac p{p_i}}\prod_{j\neq i}w_j^{\frac p{p_j}}\\
&\lesssim \prod_{\ell\in A_i} \|b_{\ell}^i\|^p_{\BMO_{\nu_i^{\theta_\ell}}}
 \sum_{\substack{P_{\ell}\in \wt{\mathcal S}, P_\ell \subset Q \\ \ell\in A_i}}\Big( \prod_{\ell\in A_i}\langle \nu_i^{\theta_{\ell}}\rangle_{P_{\ell}}^p\Big)\int_{\cap_{ \ell\in A_i}P_{\ell}} \lambda_i^{\frac p{p_i}}\prod_{j\neq i}w_j^{\frac p{p_j}} .
\end{align*}We may assume $\cap_{\ell\in A_i}P_{\ell}=P_{\ell_0}$. Then by Lemma \ref{lem:RHI} we see that
\[
\langle \nu_{i}^{\theta_{\ell_0}}\rangle_{P_{\ell_0}}^p\int_{P_{\ell_0}}  \lambda_i^{\frac p{p_i}}\prod_{j\neq i}w_j^{\frac p{p_j}}  \lesssim \int_{P_{\ell_0}} \nu_{i}^{p\theta_{\ell_0}} \lambda_i^{\frac p{p_i}}\prod_{j\neq i}w_j^{\frac p{p_j}}.
\]Since $0\le \theta_{\ell_0}\le 1$ and
\[
\nu_i^p  \lambda_i^{\frac p{p_i}}\prod_{j\neq i}w_j^{\frac p{p_j}}=\prod_{j}w_j^{\frac p{p_j}}\in A_\infty, \qquad \lambda_i^{\frac p{p_i}}\prod_{j\neq i}w_j^{\frac p{p_j}}\in A_\infty,
\]we know that
\[
 \nu_{i}^{p\theta_{\ell_0}} \lambda_i^{\frac p{p_i}}\prod_{j\neq i}w_j^{\frac p{p_j}}=\Big(\nu_i^p  \lambda_i^{\frac p{p_i}}\prod_{j\neq i}w_j^{\frac p{p_j}}\Big)^{\theta_{\ell_0}} \Big(  \lambda_i^{\frac p{p_i}}\prod_{j\neq i}w_j^{\frac p{p_j}}\Big)^{1- \theta_{\ell_0}}\in A_\infty.
\]
Therefore, we are able to sum over $P_{\ell_0}\subset\cap_{\ell\in A_i\setminus\{\ell_0\}}P_{\ell} $ and we arrive at
\[
 \sum_{\substack{P_{\ell}\in \wt{ \mathcal S}, P_\ell\subset Q\\ \ell\in A_i\setminus \{\ell_0\}}} \prod_{\ell\in A_i\setminus \{\ell_0\}}\langle \nu_i^{\theta_{\ell}}\rangle_{P_{\ell}}^p\int_{\cap_{ \ell\in A_i\setminus \{\ell_0\}}P_{\ell}}\nu_{i}^{p\theta_{\ell_0}} \lambda_i^{\frac p{p_i}}\prod_{j\neq i}w_j^{\frac p{p_j}}.
\] Iterating this process we get
\begin{align}
\int_Q  \Big(\prod_{\ell\in A_i}  |b^i_{\ell}-\langle b^i_{\ell}\rangle_Q|^p\Big) \lambda_i^{\frac p{p_i}}\prod_{j\neq i}w_j^{\frac p{p_j}}  \lesssim \prod_{\ell\in A_i} \|b_{\ell}^i\|^p_{\BMO_{\nu_i^{\theta_\ell}}}\int_Q \nu_{i}^{p\sum_{\ell \in A_i}\theta_{\ell}}  \lambda_i^{\frac p{p_i}}\prod_{j\neq i}w_j^{\frac p{p_j}}.
\end{align}
The idea is to use Proposition \ref{prop:lms} as a blackbox. It is not difficult to check that
\[
(w_1,\dots, w_{i-1}, \nu_{i}^{p_i\theta_{A_i}}  \lambda_i, w_{i+1},\dots, w_n)\in A_{\vec p},\qquad \theta_{A_i}:=\sum_{\ell \in A_i}\theta_\ell.
\]
With this observation in mind, if $\# B_i=0$, then we are done. So we may assume  $0<\# B_i\le k_i$. Similarly as the above, we write
\[
\int_Q  |f_i|\prod_{\ell\in B_i}  |b^i_{\ell}-\langle b^i_{\ell}\rangle_Q|\lesssim \prod_{\ell\in B_i} \|b_{\ell}^i\|_{\BMO_{\nu_i^{\theta_\ell}}}\sum_{\substack{P_{\ell}\in \wt{\mathcal S}, P_\ell\subset Q \\ \ell\in B_i}} \prod_{\ell\in B_i}\langle \nu_i^{\theta_{\ell}}\rangle_{P_{\ell}}\int_{\cap_{\ell\in B_i}P_{\ell}} |f_i|.
\]We may assume $\cap_{\ell\in B_i}P_{\ell}=P_{\ell_1}$, and $\cap_{\ell\in B_i\setminus \{\ell_1\}}P_{\ell}=P_{\ell_2}$ (if $\# B_i=1$ we take $P_{\ell_2}=Q$). Let $\zeta_1\in A_\infty$ which will be determined later. The point now is we can write
\begin{align*}
\sum_{\substack{P_{\ell_1}\in \wt{\mathcal S}\\ P_{\ell_1}\subset P_{\ell_2}}} \langle \nu_i^{\theta_{\ell_1}}\rangle_{P_{\ell_1}}\int_{P_{\ell_1}}|f_i|&=\sum_{\substack{P_{\ell_1}\in \wt{\mathcal S}\\ P_{\ell_1}\subset P_{\ell_2}}} \langle \nu_i^{\theta_{\ell_1}}\rangle_{P_{\ell_1}} \zeta_1(P_{\ell_1}) \langle |f_i|\zeta_1^{-1}\rangle_{P_{\ell_1}}^{\zeta_1}\\
&\lesssim \int_{P_{\ell_2}} \mathcal A_{\wt{\mathcal S}}^{\zeta_1}(f_i \zeta_1^{-1}) \nu_i^{\theta_{\ell_1}} \zeta_1,
\end{align*}where we have used Lemma \ref{lem:RHI} and
\[
\mathcal A_{\wt{\mathcal S}}^{\sigma}(h):= \sum_{P\in \wt{\mathcal S}} \langle h\rangle_Q^\sigma 1_Q.
\]
Let $h_{\ell_1}=  \mathcal A_{\wt{\mathcal S}}^{\zeta_1}(f_i \zeta_1^{-1}) \nu_i^{\theta_{\ell_1}} \zeta_1$ and $h_{\ell_k}= \mathcal A_{\wt{\mathcal S}}^{\zeta_{k}}(h_{\ell_{k-1}} \zeta_k^{-1}) \nu_i^{\theta_{\ell_k}} \zeta_k$, where again, $\zeta_2,\dots, \zeta_{\# B_i}\in A_\infty$ weights which will be determined later. By iterating this process we are able to write
\begin{align*}
\int_Q  |f_i|\prod_{\ell\in B_i}  |b^i_{\ell}-\langle b^i_{\ell}\rangle_Q|\lesssim \sum_{(\ell_s)}\int_Q  h_{\ell_{\# B_i}},
\end{align*}where the sum is taken over all the permutations of $B_i$. It remains to define $\zeta_k$, $1\le k\le \# B_i$. In fact, to bound $\|h_{\ell_{\# B_i}}\|_{L^{p_i}(\nu_i^{p_i\theta_{A_i}}\lambda_i)}$ correctly we need
\begin{enumerate}
\item $(\nu_i^{\theta_{\ell_{\# B_i}}}\zeta_{\# B_i})^{p_i}\nu_i^{p_i\theta_{A_i}}\lambda_i =\zeta_{\# B_i} $;
\item $(\nu_i^{\theta_{\ell_{k-1}}}\zeta_{k-1})^{p_i}\zeta_k^{1-p_i}=\zeta_{k-1}$;
    \item $\zeta_1^{1-p_i}=w_i$.
\end{enumerate}
These give us that $ \zeta_1=w_i^{1-p_i'}\in A_\infty$ and
\[
 \zeta_k= w_i^{1-p_i'} \nu_i^{p_i'\sum_{s=1}^{k-1}\theta_{\ell_s}}=(\lambda_i^{1-p_i'})^{\sum_{s=1}^{k-1}\theta_{\ell_s}}(w_i^{1-p_i'})^{1-\sum_{s=1}^{k-1}\theta_{\ell_s}}\in A_\infty,\quad 2\le k \le \# B_i.
\]The point of these conditions are that, each time they allow us to use the $L^{p_i}(\zeta_k)$ boundedness of $A_{\wt{\mathcal S}}^{\zeta_k}$, which is guaranteed by $\zeta_k\in A_\infty$. In particular, we have 
\[
  \|h_{\ell_{\# B_i}}\|_{L^{p_i}(\nu_i^{p_i\theta_{A_i}}\lambda_i)}\lesssim \|f_i\|_{L^{p_i}(w_i)}.
\]
Combining the arguments in the above, by Proposition \ref{prop:lms}
we get
\begin{align*}
 & \Big\|\sum_{Q\in \mathcal S} \Big\langle  |f_i|\prod_{\ell\in B_i}  |b^i_{\ell}-\langle b^i_{\ell}\rangle_Q|\Big\rangle_Q \Big(\prod_{j\neq i}\langle |f_j|\rangle_Q \Big) \prod_{\ell\in A_i}  |b^i_{\ell}-\langle b^i_{\ell}\rangle_Q| 1_Q\Big\|_{L^p(\lambda_i^{\frac p{p_i}}\prod_{j\neq i}w_j^{\frac p{p_j}})}\\
 &\lesssim \prod_{\ell=1}^{k_i} \|b_{\ell}^i\|_{\BMO_{\nu_i^{\theta_\ell}}} \prod_{i=1}^{n}\|f_i\|_{L^{p_i}(w_i)}.
\end{align*}
It remains to prove the case of $p>1$. In this case, by duality we reduce the problem to estimate
\[
  \sum_{Q\in \mathcal S} \Big\langle  |f_i|\prod_{\ell\in B_i}  |b^i_{\ell}-\langle b^i_{\ell}\rangle_Q|\Big\rangle_Q \Big(\prod_{j\neq i}\langle |f_j|\rangle_Q \Big) \Big\langle |f_{n+1}|\prod_{\ell\in A_i}  |b^i_{\ell}-\langle b^i_{\ell}\rangle_Q|\Big\rangle_Q |Q|,
\]where $f_{n+1}\in L^{p'}\big((\lambda_i^{\frac{p}{p_i}}\prod_{j\neq i}w_j^{\frac p{p_j}})^{1-p'}\big)$. Then one may deal with
\[
  \Big\langle |f_{n+1}|\prod_{\ell\in A_i}  |b^i_{\ell}-\langle b^i_{\ell}\rangle_Q|\Big\rangle_Q
\]similarly as the above. In fact, if $\# A_i=0$ then we are done. If $\# A_i>0$ then the point now is we can write
\[
  \Big\langle |f_{n+1}|\prod_{\ell\in A_i}  |b^i_{\ell}-\langle b^i_{\ell}\rangle_Q|\Big\rangle_Q \lesssim \sum_{(\ell_t)}\langle g_{\ell_{\# A_i}}\rangle_Q,
\]where the sum is taken over all permutations of $A_i$ and
\[
  g_{\ell_1}=\mathcal A_{\wt{\mathcal S}}^{\eta_1}(f_{n+1}\eta_1^{-1})\nu_i^{\theta_{\ell_1}}\eta_1,\qquad g_{\ell_k}=\mathcal A_{\wt{\mathcal S}}^{\eta_k}(g_{\ell_{k-1}}\eta_k^{-1})\nu_i^{\theta_{\ell_k}}\eta_k
\]with $\eta_1=\lambda_i^{\frac{p}{p_i}}\prod_{j\neq i}w_j^{\frac p{p_j}}\in A_\infty$ and for $2\le k\le  \#A_i$
\[ \eta_k=\nu_i^{p\sum_{t=1}^{k-1}\theta_{\ell_t}}\lambda_i^{\frac{p}{p_i}}\prod_{j\neq i}w_j^{\frac p{p_j}}= \Big(\prod_{i=1}^{n} w_i^{\frac p {p_i}}\Big)^{\sum_{t=1}^{k-1}\theta_{\ell_t}} \Big( \lambda_i^{\frac{p}{p_i}}\prod_{j\neq i}w_j^{\frac p{p_j}}\Big)^{1-\sum_{t=1}^{k-1}\theta_{\ell_t}}\in A_\infty.
\]In particular,
\[
  (\nu_i^{\theta_{\ell_{\# A_i}}} \eta_{\# A_i})^{p'}\Big(\nu_{i}^{p\sum_{\ell \in A_i}\theta_{\ell}}  \lambda_i^{\frac p{p_i}}\prod_{j\neq i}w_j^{\frac p{p_j}}\Big)^{1-p'}= \eta_{\# A_i}
\]
and \[
  \|g_{\ell_{\# A_i}}\|_{L^{p'}\Big(\big(\nu_{i}^{p\sum_{\ell \in A_i}\theta_{\ell}}  \lambda_i^{\frac p{p_i}}\prod_{j\neq i}w_j^{\frac p{p_j}}\big)^{1-p'}\Big)}\lesssim \|f_{n+1}\|_{L^{p'}\big((\lambda_i^{\frac{p}{p_i}}\prod_{j\neq i}w_j^{\frac p{p_j}})^{1-p'}\big)}.
\]Hence by applying H\"older's inequality and Proposition \ref{prop:lms} with the multiple weight
\[
  (w_1,\dots, w_{i-1}, \nu_{i}^{p_i\theta_{A_i}}  \lambda_i, w_{i+1},\dots, w_n)\in A_{\vec p}
\]concludes the proof.

\begin{rem}
  The reason why we do not track the constant above is that, we use Lemma \ref{lem:RHI} frequently which makes the presentation of the constant already very complicated, taking into account that there is even permutation involved. However, we would like to comment that the method used in the above improves the known bound obtained through the Cauchy integral trick if $\nu_i=1$. In fact, if $\nu_i=1$, for simplicity denote
  \[
    w=\prod_{i=1}^{n}w_i^{\frac p{p_i}}=\lambda_i^{\frac p{p_i}}\prod_{j\neq i} w_j^{\frac p{p_j}}.
  \]Then apart from the constant from Proposition \ref{prop:lms}, the constant produces in the case $p\le 1$ is
  \begin{align*}
    [w]_{A_\infty}^{\# A_i} \| A_{\wt{\mathcal S}}^{w_i^{1-p_i'}}\|_{L^{p_i}(w_i^{1-p_i'})}^{\# B_i}\le C [w]_{A_\infty}^{\# A_i} [w_i^{1-p_i'}]_{A_\infty}^{\# B_i}\le C ([w]_{A_\infty}+[w_i^{1-p_i'}]_{A_\infty})^{k_i}.
  \end{align*}
  And for the case $p>1$, we can track that the related constant is
  \[
    \| A_{\wt{\mathcal S}}^{w_i^{1-p_i'}}\|_{L^{p_i}(w_i^{1-p_i'})}^{\# B_i}\| A_{\wt{\mathcal S}}^{w}\|_{L^{p'}(w)}^{\# A_i}\le  C ([w]_{A_\infty}+[w_i^{1-p_i'}]_{A_\infty})^{k_i}.
  \]
  In both cases we have used the fact that
  \[
    \| A_{\mathcal S}^\sigma \|_{L^p(\sigma)}\le C [\sigma]_{A_\infty}, \qquad 1<p<\infty.
  \]Indeed, for $f, g\ge 0$
  \begin{align*}
  \langle A_{\mathcal S}^\sigma f, g\sigma \rangle =\sum_{Q\in \mathcal S}\langle f\rangle_Q^\sigma \langle g \rangle_Q^\sigma \sigma(Q)&\le \sum_{Q\in \mathcal S}\Big(\bla (M_d^\sigma f M_d^\sigma g )^{\frac 12}\bra_Q^\sigma\Big)^2 \sigma(Q)\\&\le c[\sigma]_{A_\infty}\| M_d^\sigma f M_d^\sigma g\|_{L^1(\sigma)}\le c[\sigma]_{A_\infty}\|f\|_{L^p(\sigma)}\|g\|_{L^{p'}(\sigma)} ,
  \end{align*}where we have used Carleson's embedding theorem in the second inequality due to the fact that
  \[
    \sum_{Q\in \mathcal S, Q\subset R}\sigma(Q)\le c [\sigma]_{A_\infty}\sigma(R).
  \]
  Now we see that both cases improve the bounds obtained through the Cauchy integral trick (see \cite[Theorem 5.1]{DHL} and \cite[Theorem 5.6]{BMMST}) since
  \begin{align*}
   [w]_{A_\infty}+[w_i^{1-p_i'}]_{A_\infty}\le \begin{cases}
   [w]_{A_\infty}+\sum_{i=1}^n [w_i^{1-p_i'}]_{A_\infty}, & \\
   c[\vec w]_{A_{\vec p}}^{\max \{1, \frac{p_1'}{p},\cdots, \frac{p_n'}{p}\}}. &
                                               \end{cases}
  \end{align*}
  The first inequality is trivial, for the second, it is recorded in \cite[Lemma 3.1]{DLP} that
  \[
   [w_i^{1-p_i'}]_{A_\infty} \le [\vec w]_{A_{\vec p}}^{\frac {p_i'}p},
  \]together with \cite[Lemma 3.1]{LMS} we have $[w]_{A_\infty}\le [\vec w]_{A_{\vec p}}$ (actually this can be checked directly). The general logic is, our method only involves two $A_\infty$ constants which connect to the definition of the commutator, while the Cauchy integral trick needs to involve all $A_\infty$ constants.
\end{rem}

We complete this section by the following general version of Theorem \ref{thm:main1}, whose proof is similar as the above and hence we omit the details.
\begin{thm}
Let  $T$ be an $n$-linear Calder\'on-Zygmund operator and $\mathcal I\subset \{1,\dots, n\}$. Let $C_{b_{\mathcal I}}^{k_{\mathcal I}}(T)$ be defined as in \eqref{eq:itercommu}. Assume that for any $i\in \mathcal I$, $\theta_1^i,\dots, \theta_{k_i}^i\in [0,1]$ such that $\sum_{\ell=1}^{k_i}\theta_{\ell}^i=1$, and let $\theta^i=\max\{\theta^i_\ell\}_{1\le \ell\le k_i}$. Let $1<p_1,\dots, p_n<\infty$ and $\frac 1p=\sum_{i=1}^n \frac1{p_i}$. Assume that $(u_1,\dots, u_n)\in A_{\vec p}$, where
\begin{align*}
u_i=\begin{cases}
w_i, & \textup{if} \,\, i\notin \mathcal I\\
w_i \,\,\textup{or}\,\, \lambda_i& \textup{if} \, i\in \mathcal I.
\end{cases}
\end{align*}
If  $\nu_i^{\theta^i} :=w_i^{\frac {\theta^i}{p_i}}\lambda_i^{-\frac {\theta^i}{p_i}}\in A_\infty$ and $b^i_{\ell}\in \BMO_{\nu_i^{\theta^i_{\ell}}}$ for every $1\le \ell\le k_i$ and every $i\in \mathcal I$, then
  \begin{align*}
  \Big\| C_{b_{\mathcal I}}^{k_{\mathcal I}}(T): L^{p_1}(w_1)\times \cdots\times L^{p_n}(w_n)\to L^p\big( (\prod_{i\in \mathcal I}\lambda_i^{\frac p{p_i}})\prod_{j\notin \mathcal I  }w_j^{\frac p{p_j}} \big)\Big\|\lesssim \prod_{i\in \mathcal I} \prod_{\ell=1}^{k_i} \|b_{\ell}^i\|_{\BMO_{\nu_i^{\theta^i_\ell}}}.
  \end{align*}
\end{thm}

\section{Proof of Theorem \ref{thm:main2}}\label{sec:proofmain2}
This section is devoted to proving Theorem \ref{thm:main2}. The idea is to follow the median method in \cite{Hyt18} and developed in \cite{LMV}.
Notice that the non-degeneracy of the kernel \eqref{eq:non-degen} implies that,
for any cube $Q$, we may find $\wt Q$ such that $\ell (Q)=\ell(\wt Q)$ and $\dist (Q, \wt Q)\ge C_0 \ell(Q)$ and there exists some $\sigma \in \C$ with $|\sigma|=1$ such that for all $x\in \wt Q$ and $y_1, \dots, y_n\in Q$ there holds
\[
\Re \sigma K(x, y_1,\dots, y_n)\gtrsim |Q|^{-n}.
\]
As we have mentioned in Remark \ref{rem:weakerassumption}, we will prove the same result under a weaker boundedness assumption. That is, instead of assuming
\[
  \Big\| C_{b}^{k_i}(T): L^{p_1}(w_1)\times \cdots\times L^{p_n}(w_n)\to L^p\big( \lambda_i^{\frac p{p_i}}\prod_{j\neq i  }w_j^{\frac p{p_j}} \big)\Big\|<\infty,
\]we assume 
\begin{equation}\label{eq:weakerbound}
 \sup_{A\subset Q} \frac 1{ \prod_{i=1}^n \sigma_i(Q)^{\frac 1{p_i}}} \Big\| 1_{\wt Q}  C_{b}^{k_i}(T)(1_{Q} \sigma_1, \dots, 1_A \sigma_i,\dots, 1_Q \sigma_n) \Big\|_{L^{p,\infty}( \lambda_i^{\frac p{p_i}}\prod_{j\neq i  }w_j^{\frac p{p_j}} )},
\end{equation}where $\sigma_i=w_i^{1-p_i'}$, $1\le i\le n$.
For arbitrary $\alpha \in \R$ and $x\in \wt Q \cap \{b\ge \alpha \}$, we have
\begin{align*}
\frac{1}{|Q|} \int_Q& (\alpha-b)_+^{k_i} \sigma_i \prod_{j\neq i}\frac{\sigma_j(Q)}{|Q|} \\&\le \Re \sigma \int_{Q\cap \{b\le \alpha \}}\int_{Q^{n-1}} (b(x)-b(y_i))^{k_i} K(x,y_1,\dots, y_n) \prod_{j=1}^{n}\sigma_j (y_j) dy_j.
\end{align*}
Then let $\alpha $ be a median of $b$ on $\wt Q$, since $w\in A_\infty$ we have
\[
\big| \wt Q \cap \{b\ge \alpha \}\big|\sim |\wt Q|\Rightarrow v(\wt Q \cap \{b\ge \alpha \})\sim v(\wt Q)\sim v(Q),
\]
where the last $\sim$ holds since $v=\lambda_i^{\frac p{p_i}}\prod_{j\neq i  }w_j^{\frac p{p_j}}\in A_\infty$ is doubling. Hence we have
\begin{align*}
 &v(Q)^{\frac 1p}\frac{1}{|Q|} \int_Q (\alpha-b)_+^{k_i} \sigma_i \prod_{j\neq i}\frac{\sigma_j(Q)}{|Q|}\\
 &\lesssim  \left\|  \Re \sigma \int_{Q\cap \{b\le \alpha \}}\int_{Q^{n-1}} (b(x)-b(y_i))^{k_i} K(x,y_1,\dots, y_n) \prod_{j=1}^n\sigma_j (y_j) dy_j \right\|_{L^{p,\infty}( \wt Q \cap \{b\ge \alpha\}; v)}\\
 &\lesssim \prod_{i=1}^n \sigma_i (Q)^{\frac 1{p_i}}.
\end{align*}
Plugging the fact that
\[
\Big(\frac{v(Q)}{|Q|}\Big)^{\frac 1p} \Big(\frac{\eta_i(Q)}{|Q|}\Big)^{\frac 1{p_i'}}\prod_{j\neq i} \Big( \frac{\sigma_j(Q)}{|Q|}\Big)^{\frac 1{p_j'}}\ge 1,
\]where $\eta_i=\lambda_i^{1-p_i'}$, 
we arrive at
\[
  \frac{1}{\eta_i(Q)^{\frac 1{p_i'}}\sigma_i(Q)^{\frac 1{p_i}}} \int_Q (\alpha-b)_+^{k_i} \sigma_i\lesssim 1.
\]By H\"older's inequality we then get 
\[
  \fint_Q  (\alpha-b)_+ \sigma_i^{\frac 1{k_i}}\lesssim \langle \eta_i\rangle_Q^{\frac 1{k_ip_i'}} \langle \sigma_i\rangle_Q^{\frac 1{k_ip_i}}\lesssim \langle \nu_i^{\frac 1{k_i}} \sigma_i^{\frac 1{k_i}} \rangle_Q,
\]where in the last inequality we have used Lemma \ref{lem:RHI}. 
Likewise, we also have the symmetrical estimate
\[
\fint_Q  (b-\alpha)_+ \sigma_i^{\frac 1{k_i}}\lesssim \langle \nu_i^{\frac 1{k_i}} \sigma_i^{\frac 1{k_i}} \rangle_Q.
\]
Hence \[
 \frac 1{\nu_i^{\frac 1{k_i}} \sigma_i^{\frac 1{k_i}}(Q)} \int_Q  |b-\alpha| \sigma_i^{\frac 1{k_i}}\lesssim 1.
\]
This is almost done, however, we need to take some care here. We  first notice that the above estimate gives us that $b\in \wt{ \BMO}_{\nu_i^{1/{k_i}} }(\sigma_i^{1/{k_i}})$. This step is simply seen by using triangle inequality. Then we conclude that $b\in \BMO_{\nu_i^{1/{k_i}}}$ by Lemma \ref{lem:equivalence} and Remark \ref{rem:altdef}.

\section{Further discussions}\label{sec:rough}
In the last section we provide a useful idea to deal with the sparse form related with commutators. Our start point is, given $\mu,\lambda\in A_p$ and $\nu=\mu^{\frac 1p}\lambda^{-\frac 1p}$,  the upper bound of
\[
  \sum_{Q\in \mathcal S} \bla  |b-b_Q|^s |f|^s\bra_Q^{\frac 1s} \langle |g|\rangle_Q |Q|\le C \|b\|_{\BMO_\nu} \|f\|_{L^p(\mu)}\|g\|_{L^{p'}(\lambda^{1-p'})},\qquad s>1,
\] which is left open in \cite{LORR}. The main point is, the average of 
$
  |b-b_Q| f
$ is not $L^1$ but rather some $L^s$ with $s>1$. In the first sight this may produce some trouble, as if we use 
\[
  |b-b_Q|1_Q \lesssim \sum_{P\in \mathcal S, P\subset Q}\langle |b-b_P|\rangle_P 1_P\le \|b\|_{\BMO_\nu} \sum_{P\in \mathcal S, P\subset Q}\langle \nu\rangle_P 1_P,
\]one can not handle it as before due to the power $s$. However, the new idea here is to use the Fefferman-Stein type inequality. In the classical case, the optimal result is due to P\'erez \cite{P94}.  Here we use the corresponding result for sparse operators, see \cite{B14,LPRRR}. For our purpose we record it in the following way:
\begin{equation}\label{eq:fs}
  \| A_{\mathcal S} f\|_{L^p(w)}\le C p' (r')^{\frac 1{p'}} \|f\|_{L^p(M_r w)}.
\end{equation}
Applying  \eqref{eq:fs} we have 
\begin{align*}
   \bla  |b-b_Q|^s |f|^s\bra_Q^{\frac 1s}\lesssim s' (r')^{\frac 1{s'}} \|b\|_{\BMO_\nu}\langle \nu^{s} M_r(|f|^s)\rangle_Q^{\frac 1s}=s' (r')^{\frac 1{s'}} \|b\|_{\BMO_\nu}\langle \nu^{s} M_{rs}(f)^s\rangle_Q^{\frac 1s} .
\end{align*} Then the proof is done by using the well-known bound  of $(s,1)$-sparse form (see \cite{LPRRR}) and the boundedness of $M_{rs}$ on $L^p(\mu)$ provided that $rs$ is close to $1$ enough so that $\mu\in A_{p/{sr}}$. Without providing more details we record the following result. 
\begin{thm}
Let $\Omega \in L^\infty(\mathbb S^{n-1})$ satisfy $\int_{\mathbb S^{n-1}}\Omega=0$. Suppose that $\mu, \lambda\in A_p$ and let $\nu=\mu^{\frac 1p}\lambda^{-\frac 1p}$. If $b\in \BMO_{\nu^{1/k}}$ for some $k\in \Z_+$, then 
\[
\|C_b^k(T_\Omega)\|_{L^p(\mu)\to L^p(\lambda)}\lesssim \|b\|_{\BMO_{\nu^{1/k}}}^k.
\]
\end{thm}

In the multilinear case, the idea is essentially by combining the idea in Section \ref{sec:proofmain1} and the above. However, we are not able to prove the case when different BMO functions are paired with the same function yet. Nevertheless, we show that
our method works for, e.g. 
\begin{align*}
  \sum_{Q\in \mathcal S} &\bla  |b-b_Q|^{ks} |f|^s\bra_Q^{\frac 1s} \langle |g|^s\rangle_Q^{\frac 1s} \langle |h|^s\rangle_Q^{\frac 1s}|Q|\\&\lesssim \|b\|_{\BMO_{\nu^{1/k}}}^k \|f\|_{L^{p_1}(w_1)}\|g\|_{L^{p_2}(w_2)}\|h\|_{L^{p'}((\lambda_1^{\frac p{p_1}} w_2^{\frac p{p_2}})^{1-p'})},
\end{align*}
where $(w_1, w_2)\in A_{\vec p}$, $\nu^{\frac 1k}=w_1^{\frac 1{kp_1}}\lambda^{-\frac 1{kp_1}}\in A_\infty$ and again, $s>1$ can be taken to be very close to $1$.
Now the idea is, we write 
\begin{align*}
 |b-b_Q|^k1_Q& \lesssim  \|b\|_{\BMO_{\nu^{1/k}}}^k\Big( \sum_{P\in \mathcal S, P\subset Q}\langle \nu^{\frac 1k}\rangle_P 1_P\Big)^k\\& \lesssim
 \|b\|_{\BMO_{\nu^{1/k}}}^k\Big( \sum_{P\in \mathcal S, P\subset Q}\langle \nu^{\frac 1k}\rangle_P^{\sigma} 1_P\Big)^k,
\end{align*}where $\sigma \in A_\infty$ will be determined later and we have used Lemma \ref{lem:RHI} in the second inequality. Then it is a matter to formulate a weighted version of \eqref{eq:fs}, say
\begin{equation}\label{eq:fsw}
  \| A_{\mathcal S}^\sigma  f\|_{L^p(w\sigma )}\le C  \|f\|_{L^p(\sigma M_{r,\sigma}^d w)}.
\end{equation}
The proof of \eqref{eq:fsw} is quite similar as the proof of \eqref{eq:fs} stated in \cite{LPRRR} and hence we omit the details. Now applying  \eqref{eq:fsw} with $\sigma=w_1^{1-\big(\frac {p_1}s\big)'}$, which is in $A_\infty$ with suitable $s$, we have
\begin{align*}
\bla  |b-b_Q|^{ks} |f|^s\bra_Q^{\frac 1s}&\lesssim \|b\|_{\BMO_{\nu^{1/k}}}^k \big\langle \nu^s M_{r,\sigma}^d(|f|^s \sigma^{-1}) \sigma  \big\rangle_Q^{\frac 1s}\\
&=\|b\|_{\BMO_{\nu^{1/k}}}^k \big\langle \nu^s M_{rs,\sigma}^d(|f| \sigma^{-\frac 1s})^s \sigma  \big\rangle_Q^{\frac 1s}.
\end{align*}
Then it is a matter of using the weighted estimates for $(s,s,s)$ sparse form (this is easily seen by the open property of $A_{\vec p}$) and the boundedness of $ M_{rs,\sigma}^d$ on $L^{p_1}(\sigma)$, together with the following two facts:
\[
  \nu^{p_1} \sigma^{\frac {p_1}s} \lambda_1= \sigma,\qquad \mbox{and}\qquad \sigma^{1-\frac{p_1}{s}}=w_1.
\]In the above we do not track the dependence on the weight constants, and we encourage the interested readers to do so.

 \end{document}